\documentclass[11pt]{article}
\usepackage{latexsym,amsmath}
\usepackage{version}
\usepackage{amsmath}
\usepackage{graphicx}
\usepackage{amsthm}
\usepackage{amssymb}
\usepackage{alltt}
\usepackage{epic}
\usepackage{eepic}
\usepackage{fancyheadings}
\usepackage{shadow}
\usepackage{array}
\usepackage{enumerate}
\usepackage[normalem]{ulem}
\usepackage{epsfig}
\usepackage{multicol}
\usepackage{ifthen}
\usepackage{bbold}

\font\Cp = msbm10

\newcommand{\Rrr}{\hbox{\Cp R}}

\newcommand{\binomial}[2]{\genfrac{(}{)}{0pt}{}{#1}{#2}}

\newtheorem{theorem}{Theorem}[section]
\newtheorem{proposition}[theorem]{Proposition}
\newtheorem{lemma}[theorem]{Lemma}

\newtheorem{conjecture}[theorem]{Conjecture}

\newtheorem{example}[theorem]{Example}
\newtheorem{definition}[theorem]{Definition}

\makeatletter
\@addtoreset{equation}{section}
\makeatother

\newcommand{\pair}[2]{\left\langle#1, #2\right\rangle}

\newcommand{\bfo}{{\mathbf 1}}
\newcommand{\Comb}{\mbox{\rm Comb}}
\newcommand{\onethingatopanother}[2]{\genfrac{}{}{0pt}{}{#1}{#2}}

\begin{document}

\title{Asymptotics of the Euler number of bipartite graphs}

\author{{\sc Richard EHRENBORG}
        and
        {\sc Yossi FARJOUN}}

\date{}
\maketitle

\begin{abstract}
We define the Euler number of a bipartite graph on $n$ vertices to
be the number of labelings of the vertices with $1, \ldots, n$ such that
the vertices alternate in being local maxima and local minima.
We reformulate the problem of computing the Euler number of
certain subgraphs of the Cartesian product of a graph $G$ with the
path $P_{m}$ in terms of self adjoint operators.
The asymptotic expansion of the Euler number is given in terms
of the eigenvalues of the associated operator.
For two classes of graphs, the comb graphs and the Cartesian product
$P_{2} \Box P_{m}$,
we numerically solve the eigenvalue problem.
\end{abstract}

\section{Introduction}
\label{section_introduction}

Let $G = (V, E)$ be a bipartite graph 
on $n$ vertices 
with the vertex decomposition
$V = V_{1} \cup V_{2}$,
that is, 
each edge in $G$ has one vertex in $V_{1}$
and the other in $V_{2}$.
An {\em alternating labeling} $\pi$
is a bijection
$\pi : V \longrightarrow \{1, \ldots, n\}$
such that for two adjacent vertices $u \in V_{1}$
and $v \in V_{2}$ we have that $\pi(u) < \pi(v)$.
Another way to phrase this condition is that
every vertex $u$ in $V_{1}$ is a local minimum of the bijection~$\pi$
and
every vertex $v$ in $V_{2}$ is a local maximum.
Define the Euler number ${\cal E}(G)$ to be the number
of alternating labelings $\pi$ of the vertices of the graph.

Two examples are in order.
First, for the path $P_{n}$ on $n$ vertices
the Euler number ${\cal E}(P_{n})$ is the number of alternating permutations,
that is,
the classical Euler number $E_{n}$.
Second, for a cycle $C_n$ of even length $n$ we have that
${\cal E}(C_{n}) = n/2 \cdot E_{n-1}$.
Since there are $n/2$ possible positions for
the largest label~$n$, the labeling
reduces to the path $P_{n-1}$.

Observe that we cannot drop the condition that the graph
$G$ is bipartite, since the labeling can not alternate along
an odd cycle. 
Alternatively, for a non-bipartite graph $G$
let ${\cal E}(G) = 0$.
Observe that the definition of the Euler number
is independent of the order of $V_{1}$ and $V_{2}$.
We also have the trivial lower bound
\begin{equation}
   |V_{1}|! \cdot |V_{2}|!  \leq  {\cal E}(G)  ,
\label{inequality_lower_bound}
\end{equation}
by assigning $V_{2}$ the $|V_{2}|$ largest labels.
Equality in~(\ref{inequality_lower_bound})
is only obtained for the complete bipartite graphs.
Moreover, extending the classic 
``Multiplication Theorem'' due to
MacMahon~\cite[Article~159]{MacMahon},
for the disjoint union of two graphs $G$ and $H$
we have
\begin{equation}
      {\cal E}(G \cup H)
   = 
        \binomial{m+n}{n}
      \cdot
        {\cal E}(G)
      \cdot
        {\cal E}(H)  ,
\label{equation_MacMahon}
\end{equation}
where $G$ and $H$ have $m$ respectively $n$ vertices.

Our interest is to study subgraphs of the Cartesian product
of two graphs.
For graphs
$G$ and $H$,
and $S$ a subset of vertices of $G$,
define
the product
$G \Box_{S} H$
as the graph on the vertex set $V(G) \times V(H)$
where two vertices $(u,u')$ and $(v,v')$ are adjacent in $G \Box_{S} H$
if either
$u = v \in S$ and $u'$ is adjacent to $v'$,
or
$u$ is adjacent to $v$ and $u' = v'$.
The Cartesian product $G \Box H$
is obtained as a special case of the product
$G \Box_{S} H$ with  $S=V(G)$.

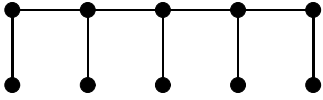
\begin{figure}
\setlength{\unitlength}{1.0mm}
\begin{center}
\begin{picture}(40,10)(0,0)
\multiput(0,0)(10,0){5}{\put(0,0){\circle*{2}}
                        \put(0,10){\circle*{2}}
                        \put(0,0){\line(0,1){10}}}
\put(0,10){\line(1,0){40}}
\end{picture}
\end{center}
\caption{The comb graph $\Comb_{m}$ as the product
$P_{2} \Box_{\{1\}} P_{m}$.}
\label{figure_comb}
\end{figure}

For the general problem we obtain the following asymptotics.

\begin{theorem}
Let $G$ be a bipartite graph on $n$ vertices
and $S$ a non-empty subset of the vertices of the graph $G$.
Then 
there exist three positive real numbers $\lambda$, $\mu$ and $c$
such that $\lambda > \mu$ and
$$     \frac{{\cal E}(G \Box_{S} P_{m})}{(m \cdot n)!}
     =
        c
     \cdot
       \lambda^{m-1} + O(\mu^{m-1})   
\:\:\:\: \:\:\:\: \mbox{ as } m \longrightarrow \infty .  $$
\label{theorem_asymptotic}
\end{theorem}

\section{The self adjoint operator}
\label{section_operator}

For a bipartite graph $G = (V,E)$ on $n$ vertices
define two subsets $X$ and $Y$
of the $n$-dimensional unit cube $[0,1]^{V}$ 
in $n$-dimensional Euclidean space $\Rrr^{V}$ by
\begin{eqnarray*}
X & = &
   \{ \vec{x} \in [0,1]^{V} \:\: : \:\:
                 x_{u} + x_{v} \leq 1
                 \mbox{ for }
                 \{u,v\} \in E \}    , \\
Y & = &
   \{ \vec{x} \in [0,1]^{V} \:\: : \:\:
                 x_{u} \leq x_{v}
                 \mbox{ for }
                 u \in V_{1},\, v \in V_{2},\, \{u,v\} \in E \} .
\end{eqnarray*}

\begin{lemma}
The two subsets $X$ and $Y$ have the same volume,
which is given by the Euler number of the graph $G$
divided by $n!$.
\end{lemma}
\begin{proof}
By reflecting the set $Y$ over all of the hyperplanes
of the form $x_{v} = 1/2$ where $v \in V_{2}$
we obtain the set $X$. Hence their volumes agree.
By cutting the $n$-dimensional cube 
with the hyperplanes $x_{u} = x_{v}$ for all $u,v \in V$
we obtain $n!$ simplices of the same volume.
Each simplex corresponds to a permutation
by reading the order of the coordinates of a point
in its interior.
The set $Y$ is the union of a subcollection of these
simplices corresponding to an alternating
labeling of the graph $G$.
\end{proof}

Let the $0,1$-function $\chi$ be defined
on the set $X \times X$ by
$$     \chi(\vec{x}, \vec{y})
    =
       \left\{ \begin{array}{c l}
               1 & \mbox{ if } x_{v} + y_{v} \leq 1
                   \mbox{ for all } v \in S, \\
               0 & \mbox{ otherwise. }
               \end{array} \right.      $$
\begin{definition}
Define the operator $T$ on  $L^{2}(X)$ by
$$
      T[f](\vec{x})
   =
      \int_{\vec{y} \in X} \chi(\vec{x}, \vec{y}) \cdot f(\vec{y}) \:
                                                            d\vec{y} .  
$$
\end{definition}
Since $\chi$ is symmetric, that is,
$\chi(\vec{x}, \vec{y}) = \chi(\vec{y}, \vec{x})$,
the operator $T$ is a self-adjoint Hilbert-Schmidt operator.
Thus we conclude that the spectrum of $T$ is real and discrete
with $0$ as the only possible accumulation point.
Furthermore, all the eigenvalues and
eigenfunctions of the operator are real.
Since \mbox{$0 \leq \chi(\vec{x}, \vec{y}) \leq 1$,}
the eigenvalues $\lambda$ lie in the closed 
interval $[-1,1]$.
Hence, there is a largest eigenvalue in absolute value.
Moreover, the eigenfunctions form a complete orthogonal set.

Let $\bfo$ denote the constant function with value $1$
on set $Y$. Now we have
\begin{proposition}
For a bipartite graph $G$ on $n$ vertices and $S$ a subset of
the vertices of the graph~$G$,
$$     \frac{{\cal E}(G \Box_{S} P_{m})}{(m \cdot n)!}
    =
       \pair{\bfo}{T^{m-1}[\bfo]} .  $$
\end{proposition}
\begin{proof}
Expanding the inner product and each of the $m-1$
applications of the operator $T$, we have that
\begin{eqnarray}
\pair{\bfo}{T^{m-1}[\bfo]}
  & = &
\int_{\vec{x}_{1}} T^{m-1}[\bfo](\vec{x}_{1}) \: d\vec{x}_{1}
\nonumber \\
  & = &
\int_{\vec{x}_{1}} 
\int_{\vec{x}_{2}} 
\chi(\vec{x}_{1},\vec{x}_{2})
  \cdot
T^{m-2}[\bfo](\vec{x}_{2})
\:
d\vec{x}_{2}
d\vec{x}_{1}
\nonumber \\
  & \vdots &
\nonumber \\
  & = &
\int_{\vec{x}_{1}} 
\int_{\vec{x}_{2}} 
  \cdots
\int_{\vec{x}_{m}} 
\chi(\vec{x}_{1},\vec{x}_{2})
  \cdot
\chi(\vec{x}_{2},\vec{x}_{3})
  \cdots
\chi(\vec{x}_{m-1},\vec{x}_{m})
\:
d\vec{x}_{m}
  \cdots
d\vec{x}_{2}
d\vec{x}_{1}  .
\label{equation_m_n}
\end{eqnarray}
Let $\vec{x}_{i}$ be the vector $(x_{v,i})_{v \in V}$.
Then the integral in equation~(\ref{equation_m_n})
is over all of the $m \cdot n$ variables $x_{v,i}$
with the boundary condition that
(i) $0 \leq x_{v,i} \leq 1$,
(ii) $x_{v,i} + x_{v,i+1} \leq 1$ for $1 \leq i \leq m-1$ and $v \in S$,
and
(iii) $x_{u,i} + x_{v,i} \leq 1$ for $\{u,v\}$ an edge in $G$.
These inequalities describe exactly the set
$X_{G \Box_{S} P_{m}}$ and hence the integral is given
by the ratio
${\cal E}(G \Box_{S} P_{m})/(m\cdot n)!$.
\end{proof}

\begin{theorem}
Let $G$ be a bipartite graph on $n$ vertices
and
$S$ a non-empty subset of the vertices of the graph $G$.
Then we have
$$     \frac{{\cal E}(G \Box_{S} P_{m})}{(m \cdot n)!}
     =
       \sum_{k \geq 0}
           \frac{\pair{\varphi_{k}}{\bfo}^{2}}
                {\|\varphi_{k}\|^{2}}
         \cdot
           \lambda_{k}^{m-1} ,  $$
where the eigenvalues of the operator $T$ are
$\{\lambda_{k}\}_{k \geq 0}$
and $\varphi_{k}$ is the eigenfunction associated
to the eigenvalue $\lambda_{k}$.
\label{theorem_main}
\end{theorem}
\begin{proof}
Expand the function
$\bfo$ in terms of eigenfunctions:
$$  \bfo
  =
    \sum_{k \geq 1} 
        \frac{\pair{\varphi_{k}}{\bfo}}
             {\|\varphi_{k}\|^{2}}
      \cdot
        \varphi_{k}  .  $$
Apply $T^{m-1}$ and take the inner product with $\bfo$
and the result follows.
\end{proof}

When the set $S$ is empty, Theorem~\ref{theorem_main}
is trivial. In that case
$G \Box_{\emptyset} P_{m}$ is the disjoint union
of $m$ copies of $G$. Using~(\ref{equation_MacMahon})
we have that
$$     \frac{{\cal E}(G \Box_{\emptyset} P_{m})}{(m \cdot n)!}
     =
       \left( \frac{{\cal E}(G)}{n!} \right)^{m}   .  $$

\begin{example}
{\rm
When the graph $G$ consists of a singleton vertex
and $S$ consists of this vertex,
then the product
$G \Box_{S} P_{m}$ is exactly the path on $m$ vertices,
and its Euler number is the classical $E_{m}$.
In this case the operator $T$ is given by
$$   T[f](x) = \int_{0}^{1-x} f(z) \: dz  .  $$
This operator has eigenvalues
$\lambda_{k} = 2/(\pi \cdot k)$
where $k = \ldots, -7, -3, 1, 5, 9, \ldots$
and eigenfunctions
\mbox{$\varphi_{k} = \cos(x/\lambda_{k})$}.
Calculating
$\pair{\varphi_{k}}{\bfo} = \lambda_{k}$
and
$\|\varphi_{k}\|^{2} = \pair{\varphi_{k}}{\varphi_{k}} = 1/2$
we obtain the following classical asymptotic expansion for the 
Euler number
\begin{eqnarray*}
      E_{m}
  & = &
      2 
   \cdot
      m!
   \cdot
      \sum_{k}
              \left( \frac{2}{\pi \cdot k}\right)^{m+1}    \\
  & = &
      2 
   \cdot
      m!
   \cdot
      \sum_{\onethingatopanother{j \geq 1}{j \mbox{ \scriptsize odd}}}
              \left( -1 \right)^{(m+1) \cdot (j-1)/2}  
            \cdot
              \left( \frac{2}{\pi \cdot j}\right)^{(m+1)}  ,
\end{eqnarray*}
where $j = |k|$, that is, $k = (-1)^{(j-1)/2} \cdot j$.
See~\cite[Section~4]{Ehrenborg_Levin_Readdy}.
}
\label{example_classical_Euler_number}
\end{example}

Let $W$ be the subspace of $L^{2}(X)$ consisting of the functions only
depending on the variables $x_{u}$ where $u$ belongs to~$S$.  In the
case when $S$ is the vertex set of the graph $G$ the space $W$ is
$L^{2}(X)$. The following result applies to the case when $S$ is
strictly contained in the vertex set of $G$.
\begin{proposition}
The image of the operator $T$
is contained in subspace $W$.
Hence all the eigenfunctions associated to non-zero
eigenvalues belong to $W$.
\label{proposition_subspace}
\end{proposition}
\begin{proof}
For a vertex $v$ not in the set $S$, observe that the function
$\chi(\vec{x}, \vec{y})$ does not depend on the variable~$x_{v}$.
Hence when integrating
$\chi(\vec{x}, \vec{y}) \cdot f(\vec{y})$
over all 
$\vec{y} \in X$
the resulting function $T[f]$ does not depend on~$x_{v}$,
that is, $T[f]$ belongs to the space $W$. 
The second statement follows from the defining relation
for eigenfunctions.
\end{proof}

The Frobenius-Perron result
applies to matrices, that is,
linear operators on a finite-dimensional vector space.
An operator version
of Frobenius-Perron 
was discovered by Kre\u{\i}n and Rutman~\cite{Krein_Rutman}.
We present a specialized version of their result.
Let $Z$ be a measurable space.
Recall that two functions in $L^{2}(Z)$ are considered
the same if they differ on a set of measure $0$.
We call a function $f \in L^{2}(Z)$ non-negative
if $f(x) \geq 0$ for almost all $x \in Z$.
Similarly, we call
the function $f$ positive
if $f(x) > 0$ for almost all $x \in Z$.
An operator $M$ on $L^{2}(Z)$ is {\em positivity improving}
if for all non-negative but non-zero functions $f$
the function $M[f]$ is positive.
\begin{theorem}[Kre\u{\i}n-Rutman]
Let $M$ be an operator $L^{2}(Z)$ such that
there is a positive integer $k$ so
that $M^{k}$ is positivity improving.
Then the largest eigenvalue $\lambda$ (in modulus) of $M$ 
is real, positive and simple. Moreover, the associated eigenfunction
$\varphi$ is a positive function on $Z$.
\end{theorem}

Applying Kre\u{\i}n-Rutman to our operator $T$, we have
\begin{proposition}
The operator $T^{2}$ is positivity improving.
The largest eigenvalue (in absolute value) $\lambda$ 
of the operator $T$ is real, positive and simple.
Furthermore, the associated eigenfunction $\varphi$ is positive.
\end{proposition}
\begin{proof}
Let $f$ be a non-negative, non-zero function  in $L^{2}(X)$.
By the definition of the operator $T$ we have
that in a neighborhood of $0$ the function
$T[f]$ has a positive support.
By applying the operator~$T$ again
we obtain that every point in the interior of $Y$
takes a positive value in the function $T^{2}[f]$.
The remainder of the proposition follows from
Kre\u{\i}n-Rutman.
\end{proof}

\begin{proof}[Proof of Theorem~\ref{theorem_asymptotic}]
By letting $\lambda$ be the largest eigenvalue
of the operator $T$ and letting $\mu$ be a bound
on the next largest eigenvalue such that $\lambda > \mu$,
the result follows.
\end{proof}

\section{The comb graph}

\begin{table}
$$
\begin{array}{*{6}{r }}
m & E_{m}   & {\cal E}(\Comb_{m}) & \sim{\cal E}(\Comb_{m}) &
 {\cal E}(P_{2} \Box P_{m}) & \sim{\cal E}(P_{2} \Box P_{m}) \\ \hline
 1 &     1   &               1 &  0.99379166  &             1 & 0.98451741 \\
 2 &     1   &               5 &  4.99961911  &             4 &   4.002193 \\       
 3 &     2   &              66 &  65.9990972  &            44 &   43.99713 \\       
 4 &     5   &            1613 &  1612.99965  &           896 &   896.0018 \\        
 5 &    16   &           63480 &  63480.0072  &         29392 &   29391.93 \\         
 6 &    61   &         3662697 &  3.66269757\times 10^6  &       1413792 &   1.413789 \times 10^6 \\
 7 &   272   &       291407424 &  2.91407470\times 10^8  &      93770800 &   9.377064 \times 10^7 \\  
 8 &  1385   &     30572578425 &  3.05725832\times 10^{10}  &    8201380224 &   8.201366 \times 10^9 \\  
 9 &  7936   &   4089549416832 &  4.08955006\times 10^{12}  &  914570667792 &   9.145691 \times 10^{11}\\    
10 & 50521   & 679329771871725 &  6.79329879\times 10^{14}  & 126651310675680&   1.266511 \times 10^{14}      
\end{array}
$$
\caption{Table of the Euler numbers, the
number of alternating $2 \times m$ arrays, comb graph $\Comb_{m}$, and their numerical approximations, denoted by $\sim{\cal E}(\Comb_{m})$ and  $\sim{\cal E}(P_{2} \Box P_{m})$.}
\label{tab:E_n}
\end{table}

\begin{figure}[t!]
\center{\resizebox{!}{5cm}{\includegraphics{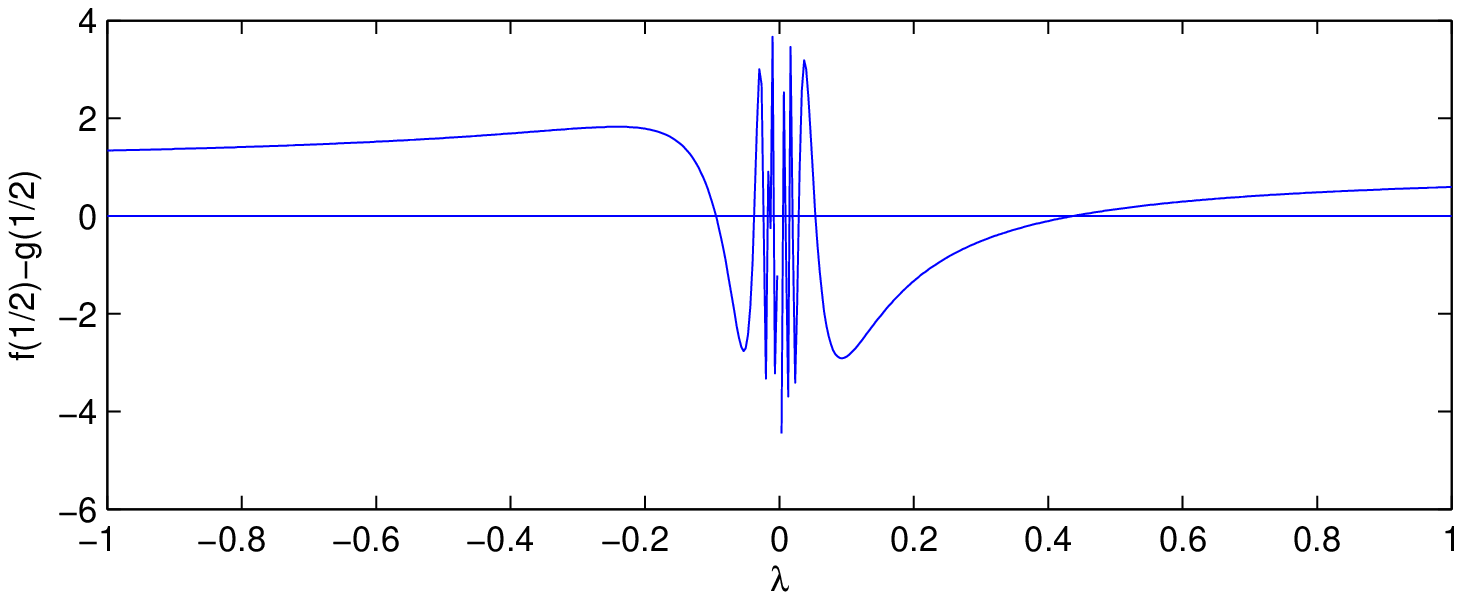}}}
\center{\resizebox{!}{5cm}{\includegraphics{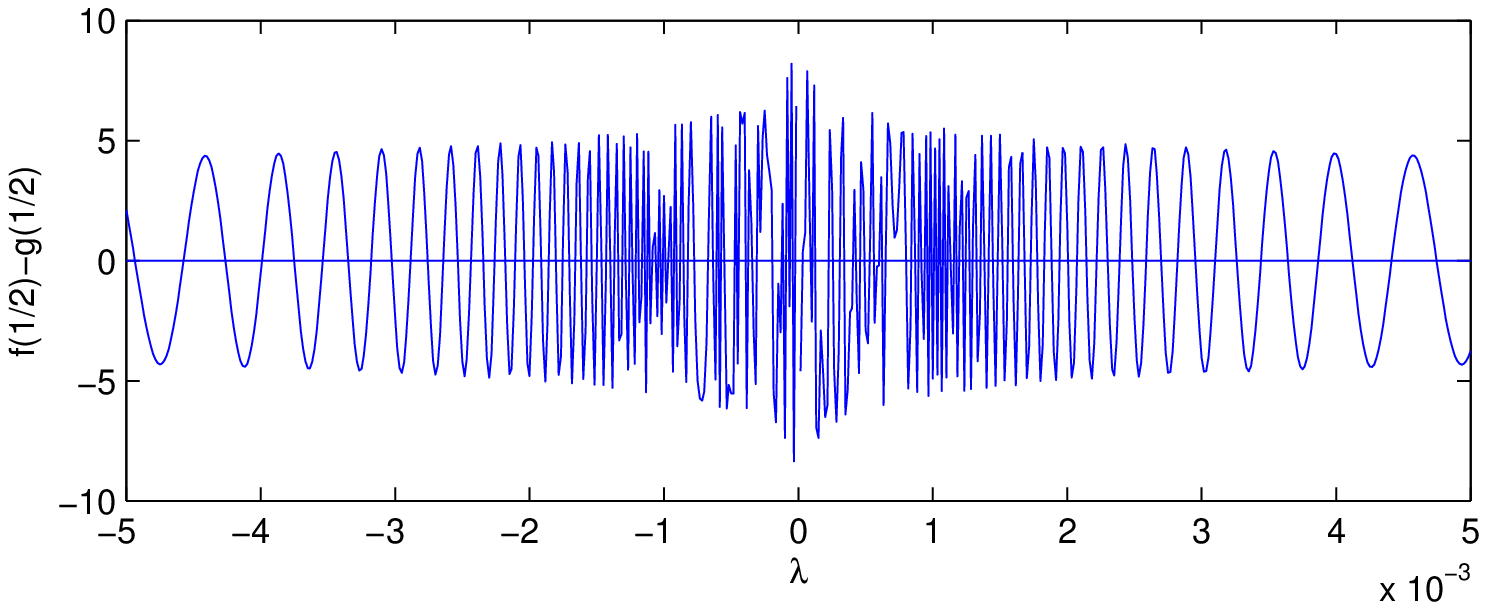}}}

\begin{minipage}{5in}
  \caption{The difference $f(1/2)-g(1/2)$ found by solving the system
    of ODE in~(\ref{equation_comb_system}) with a given value of
    $\lambda$. The roots of this plot correspond to eigenvalues. The
    lower plot is a magnification of the center domain.}
\label{fig:lambda:2}
\end{minipage}
\end{figure}

We now turn our attention to 
the comb graph. See Figure~\ref{figure_comb}.
Recall that the comb graph is defined by the product
$P_{2} \Box_{\{1\}} P_{m}$.
In this case  the space $X$ is the triangle
$$   X = \{(x,y) \:\: : \:\: x,y \geq 0,\, x+y \leq 1\}  .  $$
However, 
following
Proposition~\ref{proposition_subspace}
in order to find the eigenvalue and eigenfunctions of $T$ it is enough to consider the subspace $W$ of 
$L^{2}(X)$ consisting of functions depending only 
on the variable $x$.
Observe that $W$ inherits
the inner product
$$  \pair{f}{g}_{W}
  =
    \int_{0}^{1} (1-x) \cdot f(x) \cdot \overline{g(x)} \: dx      .  $$
Moreover, the operator $T$ is given by
\begin{equation*}
     T[f](x) = \int_{0}^{1-x} (1-z) \cdot f(z) \: dz  .  
\label{eq:comb:integral}
\end{equation*}
The next step is to find all of the eigenvalues and eigenfunctions
of the operator $T$, that is, functions $f(x)$ so that
\begin{equation}
\lambda \cdot f(x) = T[f](x)=\int_{0}^{1-x} (1-z) \cdot f(z) \: dz  . 
\label{eq:comb:eig:fun}
\end{equation}
We convert this integral equation into a differential equation
by differentiating to obtain
\begin{equation}
\lambda \cdot f'(x)     =    - x \cdot f(1-x) . 
\label{equation_comb_diff}
\end{equation}
To convert this into an ordinary differential equation (ODE),
we define $g(x) = f(1-x)$ and thus
\begin{equation}
    \left(\begin{array}{c}
        f \\
        g
          \end{array} \right)^{\prime}
  =
    \left(\begin{array}{c c}
        0 & -x/\lambda \\
        (1-x)/\lambda & 0 
          \end{array} \right)
\cdot
    \left(\begin{array}{c}
        f \\
        g
          \end{array} \right).
\label{equation_comb_system}
\end{equation}
Together with the boundary conditions
\begin{equation}
f(0)=1, \quad g(0) = 0,\quad f(1/2)=g(1/2),
\end{equation}
which come from the integral equation  \eqref{eq:comb:eig:fun} and the algebraic relationship between $f$ and $g$,
this linear system is equivalent to the original integral equation.
The condition that $f(0)=1$ is our choice of normalization for the eigenfunctions. 
The only solution which has $f(0)=g(0)=0$ is identically zero, thus this normalization is valid.
We proceed to solve this ODE numerically.

First, we solve the system from 0 to 1/2 for various values of
$\lambda$ and find the difference $f(1/2)-g(1/2)$. 
This is plotted in Figure~\ref{fig:lambda:2}. 
The roots of this plot correspond to eigenvalues
and we find them numerically. 
For each eigenvalue $\lambda$ the functions $f$ and $g$ are found and
$f(x)$ over the unit interval is reconstructed. 
The eigenfunctions corresponding to the largest (in absolute value)
eigenvalues are plotted in Figure~\ref{fig:eig:fun:2}.
\begin{figure}[t!]
\center{\resizebox{!}{8cm}{\includegraphics{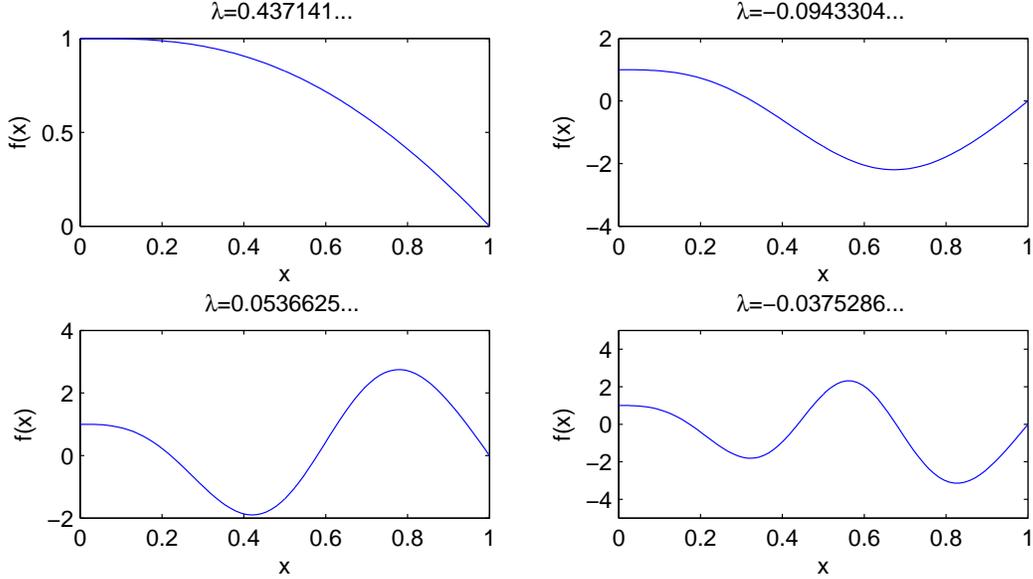}}}
\begin{minipage}{5in}
\caption{The eigenfunctions associated with the four largest (in
  absolute value) eigenvalues. 
  The function $f(x)$, in $[0, 1]$ is
  found by using $f(x)$ in $[1/2, 1]$ and $g(1-x)$ in $[0, 1/2]$.}
\label{fig:eig:fun:2}
\end{minipage}
\end{figure}
The resulting norms and constants $c_{n}$ are tabulated in Table~\ref{tab:num:vals:2}.
\begin{table}[h!]
\begin{center}
\begin{tabular}{llll}
$\lambda$ & $\pair{f(x)}{\bfo}$ & $\|f(x)\|^{2}$& $c_{n}$\\\hline
 0.437141117 & 0.437141151 &0.398916677& 0.479028320\\
 -0.094330445 & 0.094331326&   0.690741849& 0.012882380\\
0.053662538&  0.053688775&  0.829794009&0.003473735\\
 -0.037528586&  -0.037546864&   0.932757330&  0.001511397
\end{tabular}
\begin{minipage}{5in}
\caption{The values of $\lambda$,  $\pair{f(x)}{\bfo}_{W}$, $\|f(x)\|_{W}^{2}$,
  and $c_{n}$ for the first four eigenfunctions shown in
  Figure~\ref{fig:eig:fun:2}. The constant $c_{n}$ is
  the ratio $\pair{f(x)}{\bfo}_{W}^2/\|f(x)\|_{W}^{2}$
  for the $n{}^{\text{th}}$ eigenfunction. }
\label{tab:num:vals:2}
\end{minipage}
\end{center}
\end{table}

The Euler numbers for the $\Comb_{n}$ graphs are calculated from
the numerical approximation for $\lambda_{n}$ and~$c_{n}$ using the
first four terms in the series in Theorem~\ref{theorem_main}. 
They are tabulated in the fourth column of Table~\ref{tab:E_n}.

\section{Alternating $2$ by $m$ arrays}

The Euler number of the Cartesian product of two
paths $P_{m}$ and $P_{n}$ counts the number of
{\em alternating $m$ by $n$ arrays.}
That is, 
the number of
assignments of the integers $1,2, \ldots, m \cdot n$
to an $m$ by $n$ array such that
each entry is a local maximum or a local minimum.
Hence, if $i+j$ is even then
the entry $a_{i,j}$ should be less than
the four adjacent entries
$a_{i-1,j}, a_{i+1,j}, a_{i,j-1}, a_{i,j+1}$.
Similarly, if $i+j$ is odd then
the entry $a_{i,j}$ should be larger than
the four adjacent entries.

In the following we study the number
of alternating $2$ by $m$ arrays, that is,
the Euler number of the graph
$P_{2} \Box P_{m}$.
The graph $G$ is the path on two vertices $P_{2}$
and $S = \{1,2\}$.
As before, the space $X$ is the triangle
$$   X = \{(x,y) \:\: : \:\: x,y \geq 0, x+y \leq 1\}  .  $$
Observe that the operator $T$ has the form
$$
T[f](x,y)
    =
\int_{R} f(z,w) \: dz\, dw  ,
$$
where $R$ is the region described by
the inequalities
$0 \leq z \leq 1-x$,
$0 \leq w \leq 1-y$ and
$z+w \leq 1$.
Since $x + y \leq 1$,
equivalently
$(1-x) + (1-y) \geq 1$,
the inequality $z + w \leq 1$
cuts off a triangle from the rectangle
$[0,1-x] \times [0,1-y]$.
Hence we have
\begin{eqnarray}
T[f](x,y)
  & = &
\int_{0}^{y} \int_{0}^{1-y} f(z,w) \: dw\, dz
  +
\int_{y}^{1-x} \int_{0}^{1-z} f(z,w) \: dw\, dz 
\label{equation_P_2_one} \\
  & = &
\int_{0}^{x} \int_{0}^{1-x} f(z,w) \: dz\, dw
  +
\int_{x}^{1-y} \int_{0}^{1-w} f(z,w) \: dz\, dw  .
\label{equation_P_2_two} 
\end{eqnarray}

In order to study this operator $T$,
it will be easier to work in a different space.
Let $U$ be the space of functions $g(x)$
on the interval $[0,1]$ that satisfy the inequality

$$     \int_{X}
            (g(x) + g(y)) \cdot \overline{(g(x) + g(y))} \: dx\, dy
     < \infty   .  $$
Enrich the space $U$ with the following inner product
$$  \pair{g}{h}_{U}
  =
    \int_{X}
            (g(x) + g(y)) \cdot \overline{(h(x) + h(y))} \: dx\, dy
    .  $$
Define $L$ to be the linear map
$L : U \longrightarrow L^{2}(X)$
defined by $L[g](x,y) = g(x) + g(y)$.
Observe that the map $L$ preserves the inner product,
that is, $L$ is an isometry.
Moreover, $L$ is an injective map.
Furthermore, define the operator $T$ on $U$
by
\begin{equation}
T[g](x)
  =
(1-x) \cdot \int_{0}^{x} g(s) \: ds
  +
\int_{x}^{1-x} (1-s) \cdot g(s) \: ds   .
\label{equation_T_again}
\end{equation}
The reason why we denote this operator also by $T$
will be clear from the next proposition.
\begin{proposition}
The isometry $L$ and the operator $T$ commute,
that is,
$T \circ L = L \circ T$.
\end{proposition}
\begin{proof}
Apply the original operator $T$ to the function
$L[g](x,y) = g(x) + g(y)$.
We do this by applying equation~(\ref{equation_P_2_one}) to $g(x)$
and
applying equation~(\ref{equation_P_2_two}) to $g(y)$.
We then have
\begin{eqnarray*}
T[L[g]](x,y)
  & = &
\int_{0}^{y} \int_{0}^{1-y} g(z) \: dw\, dz
  +
\int_{y}^{1-x} \int_{0}^{1-z} g(z) \: dw\, dz  \\
  &   &
  +
\int_{0}^{x} \int_{0}^{1-x} g(w) \: dz\, dw
  +
\int_{x}^{1-y} \int_{0}^{1-w} g(w) \: dz\, dw   \\
  & = &
(1-y) \cdot \int_{0}^{y} g(z) \: dz
  +
\int_{y}^{1-x} (1-z) \cdot g(z) \: dz  \\
  &   &
  +
(1-x) \cdot \int_{0}^{x} g(w) \: dw
  +
\int_{x}^{1-y} (1-w) \cdot g(w) \: dw   \\
  & = &
L[T[g]](x,y)  ,
\end{eqnarray*}
where we used the fact that
$\int_{a}^{b} + \int_{c}^{d} = \int_{a}^{d} + \int_{c}^{b}$.
\end{proof}

Hence we have that
$$   \pair{\bfo}{T^{m-1}[\bfo]}_{L^{2}(X)}
   =
     \frac{1}{4} \cdot \pair{\bfo}{T^{m-1}[\bfo]}_{U}  , $$
since $L[1/2 \cdot \bfo] = \bfo$.
Thus it is enough to solve the eigenvalue problem
$\lambda \cdot g(x) = T[g](x)$ in $U$ for non-zero $\lambda$, 
that is,
\begin{equation}
   \lambda \cdot g(x) 
      =
(1-x) \cdot \int_{0}^{x} g(s) \: ds
  +
\int_{x}^{1-x} (1-s) \cdot g(s) \: ds   . 
\label{equation_g}
\end{equation}
Differentiate to obtain
\begin{equation}
   \lambda \cdot g'(x) 
      =
- \int_{0}^{x} g(s) \: ds
  -
  x \cdot g(1-x)  .  
\label{equation_g_prime}
\end{equation}
Differentiate again
\begin{equation}
   \lambda \cdot g''(x) 
      =
- g(x)
  -
  g(1-x)  
  +
  x \cdot g'(1-x)  .  
\label{equation_g_prime_prime}
\end{equation}
As before, we convert this into a linear system of differential
equations by defining
\mbox{$h(x) = g(1-x)$.}
\begin{equation}
    \left(\begin{array}{c}
        g \\
        h \\
        g' \\
        h' 
          \end{array} \right)^{\prime}
  =
    \left(\begin{array}{c c c c}
        0 & 0 & 1 & 0 \\
        0 & 0 & 0 & 1 \\
        -1/\lambda & -1/\lambda & 0 & -x/\lambda\\
        -1/\lambda & -1/\lambda & (1-x)/\lambda & 0
          \end{array} \right)
\cdot
    \left(\begin{array}{c}
        g \\
        h \\
        g' \\
        h' 
          \end{array} \right)
\label{equation_system}
\end{equation}
We solve this system differential system
numerically on the interval $[1/2,1]$.
To find boundary conditions we set \mbox{$x = 1/2$} in
equations~(\ref{equation_g})
and~(\ref{equation_g_prime}).
\begin{eqnarray}
   \lambda \cdot g(1/2) 
  & = &
1/2 \cdot \int_{0}^{1/2} g(s) \: ds  ,
\label{equation_g_half} \\
   \lambda \cdot g'(1/2) 
  & = &
- \int_{0}^{1/2} g(s) \: ds
  -
  1/2 \cdot g(1/2)  .  
\label{equation_g_prime_half}
\end{eqnarray}
\begin{figure}[t]
\center{\resizebox{!}{5cm}{\includegraphics{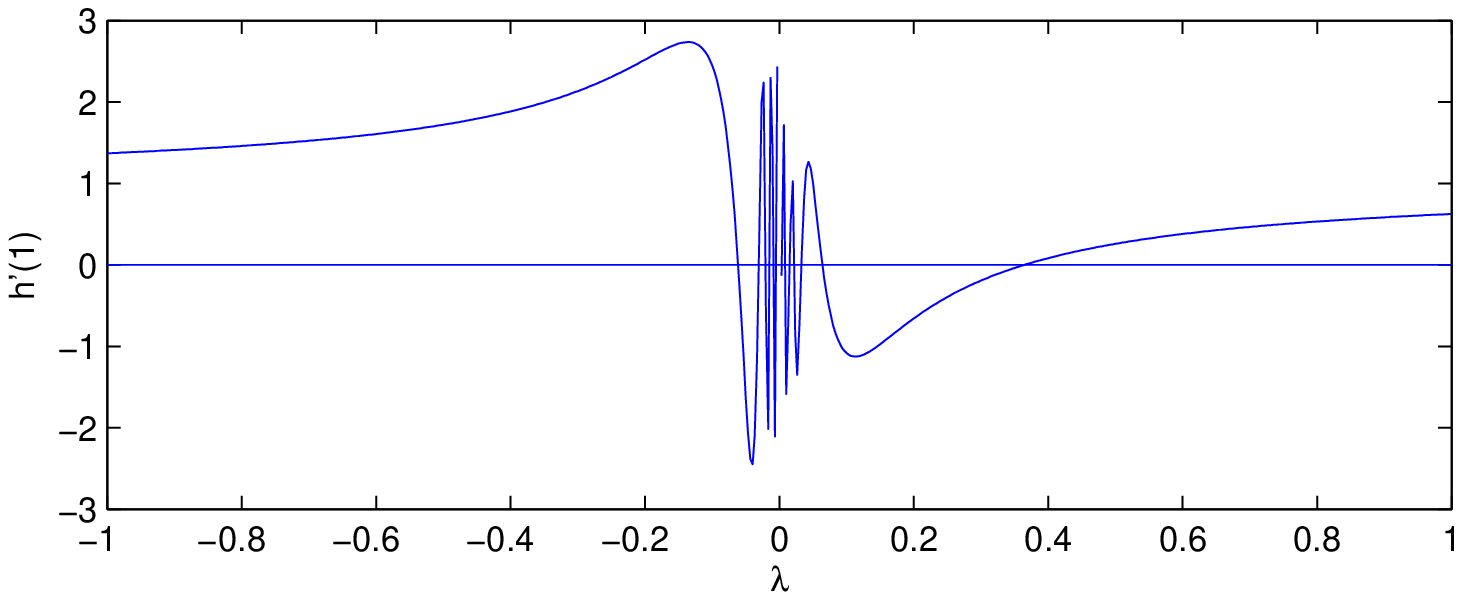}}}
\center{\resizebox{!}{5cm}{\includegraphics{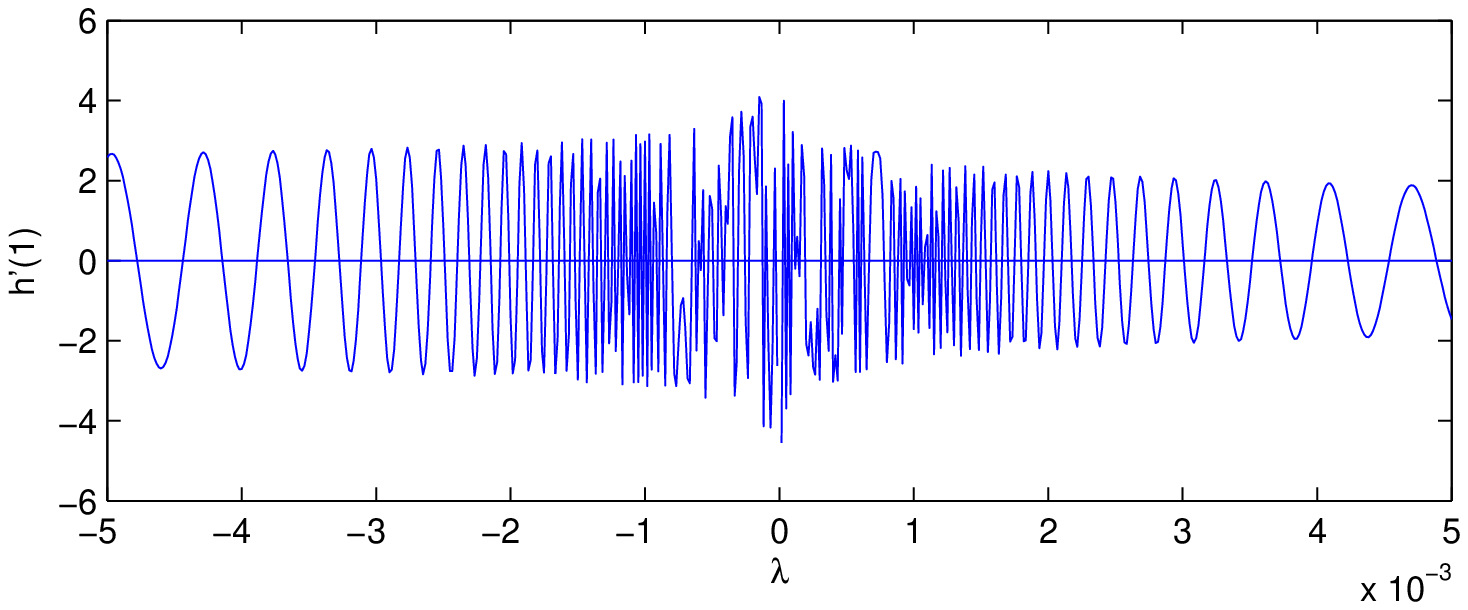}}}
\begin{minipage}{5in}
\caption{The value of $h'(1)$ found by solving the system of ODE's
  in~(\ref{equation_system})
  with a given value of $\lambda$. The roots of this plot
  correspond to eigenvalues. The lower plot is a magnification of the
  center domain.}
\label{fig:lambda}
\end{minipage}
\end{figure}
Observe that \mbox{$g(1/2) = h(1/2) = 0$,}
\eqref{equation_g_half}
and
\eqref{equation_g_prime_half}
imply
that \mbox{$g'(1/2)=h'(1/2)=0$} and therefore 
corresponds to the zero solution of~\eqref{equation_system}.
Since we are looking for the non-zero solution, we normalize such that
\begin{equation}
  g(1/2) = h(1/2) = \lambda/2 \ne 0,
\label{eq:IC:g_h}
\end{equation}
This gives us two conditions at $x=1/2$, and also
implies that
$\int_{0}^{1/2} g(x) \: dx = \lambda^{2}$.
Combined with~\eqref{equation_g_prime_half}
this gives two more conditions
\begin{equation}
  g'(1/2) = -h'(1/2) = -\lambda - \frac{1}{4}  .
\label{eq:IC:g'_h'}
\end{equation}
Thus given the parameter $\lambda$ we can solve the system from
$x=1/2$ to $x=1$. 
At $x=1$ however, the integral equation \eqref{equation_g_prime} yields another constraint:
\begin{equation}
   h'(1) = 0.
\end{equation}
Equivalently, looking at \eqref{equation_g} (and remembering that $g(x) = h(1-x)$) for $x=0,\,1$ we find that 
\begin{equation}
   h(1) = g(0) = -g(1).
\end{equation}
This is only satisfied for a discrete set of eigenvalues $\lambda$.
To find this set we (numerically) solve the ODE starting from $x=1/2$ using the
initial conditions given by~(\ref{eq:IC:g_h}, \ref{eq:IC:g'_h'}) and search values
of $\lambda$ that give $h'(1)=0$. 
Figure~\ref{fig:lambda} shows the value of $h'(1)$ for various values
of $\lambda$. 

A numerical root-finding algorithm finds the first few roots,
that is, eigenvalues $\lambda$.
The associated eigenfunctions are shown in Figure~\ref{fig:eigen:fun}.
Finally, to find $c_{n}$
we must evaluate $\pair{g}{\bfo}_U$ and $\pair{g}{g}_U$.
This is again done numerically. 
\begin{figure}[h!]
\center{\resizebox{!}{8cm}{\includegraphics{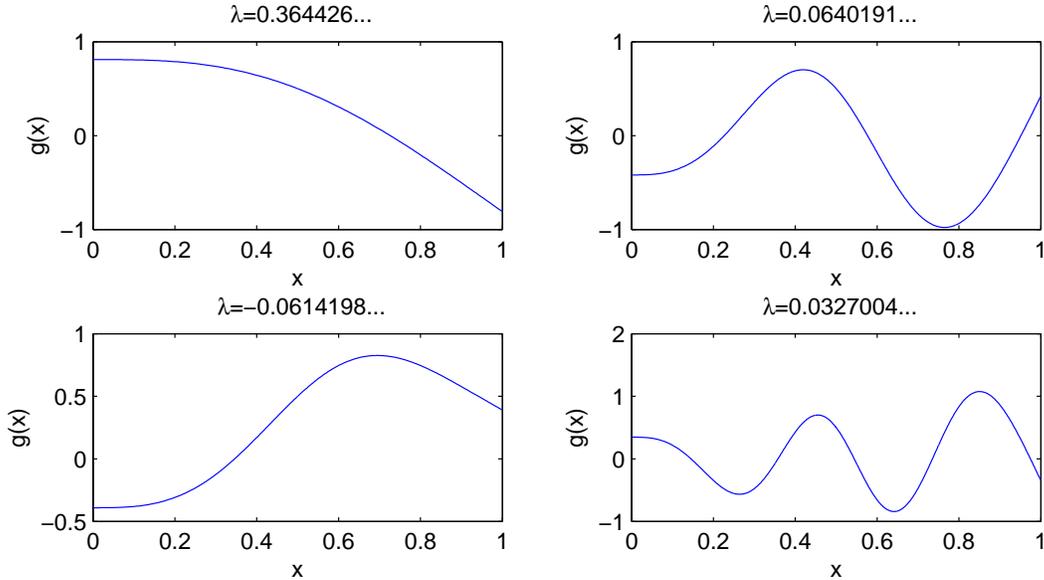}}}
\begin{minipage}{5in}
\caption{The eigenfunctions associated with the four largest (in
  absolute value) eigenvalues. 
  The whole function $g(x)$, from 0 to 1 is
  found by using $g(x)$ between $1/2$ and $1$ and $h(1-x)$ between $0$
  and $1/2$.}
\label{fig:eigen:fun}
\end{minipage}
\end{figure}
The results are shown in Table~\ref{tab:num:vals}.

\begin{table}[h!]
\begin{center}
\begin{tabular}{llll}
$\lambda$ & $\pair{g(x)}{\bfo}$ & $\|g(x)\|^{2}$& $c_{n}$\\\hline
0.364425573038 &   0.59039705924381&0.75905252379149& 0.45921550437989\\
0.064019105418 &  -0.05366486422899&0.29041608589489&0.00991652250888\\
-0.06141983509 &   0.04799387821016&0.10956972701418&0.02102234265267\\
0.03270035262  &0.02267022514715&0.24295945072711&0.00211532873771
\end{tabular}
\begin{minipage}{5in}
\caption{The values of $\lambda$,  $\pair{g(x)}{\bfo}_{U}$, $\|g(x)\|_{U}^{2}$,
  and $c_{n}$ for the first four eigenfunctions shown in
  Figure~\ref{fig:eigen:fun}. The constant $c_{n}$ is the ratio
  $\pair{g(x)}{\bfo}_{U}^2/\|g(x)\|_{U}^{2}$
  for the $n{}^{\text{th}}$ eigenfunction.}
\label{tab:num:vals}
\end{minipage}
\end{center}
\end{table}

The resulting predictions for the Euler numbers are shown
in the last column of Table~\ref{tab:E_n}.

\section{Concluding remarks}
\label{section_remarks}

Another graph to investigate is
the product with the even cycle $C_{2 m}$, that is,
${\cal E}(G \Box_{S} C_{2 m})$. We conjecture that
the resulting Euler number is 
asymptotically a constant times
the associated Euler number for the product with a path,
that is,
$$     \frac{{\cal E}(G \Box_{S} C_{2 m})}
            {{\cal E}(G \Box_{S} P_{2 m})}
     \longrightarrow 
       c    ,  $$
as $m$ tends to infinity
and $c$ is a positive constant less than $1$ for $S$ non-empty.

Does the eigenfunction $\varphi$ corresponding to the largest eigenvalue
$\lambda$ carry information about the distribution of entries
in the first copy of $G$ in an alternating labeling of $G \Box_{S} P_{m}$?
More specifically, in the case of 
alternating $2$ by $m$ arrays, let
${\cal E}(P_{2} \Box P_{m}; i,j)$
be the number of alternating arrays where
the first column consists of the two entries $i$ and $j$,
where $1 \leq i < j \leq 2m$.
Is the integer ${\cal E}(P_{2} \Box P_{m}; i,j)$
approximated by
$c \cdot (2m)! \cdot \lambda^{m-1} \cdot (g(i/2m) + g(1-j/2m))$
where $c$ is the appropriate constant
and $g$~is the first
eigenfunction displayed in
Figure~\ref{fig:eigen:fun}?

These techniques for obtaining the asymptotic behavior
of the Euler numbers 
can be used for other classes of graphs as well.
See for instance the graph $H_{m}$
in Figure~\ref{figure_hexagons}, which
is built by gluing hexagons together. 
Although
Theorem~\ref{theorem_asymptotic}
does not directly apply to this class of graphs,
one can extend the theory to obtain the same
asymptotic result.
Hence we have
$$     \frac{{\cal E}(H_{m})}
            {(4 \cdot m + 2)!}
     =
        c
     \cdot
       \lambda^{m-1} + O(\mu^{m-1})   .  $$
The essential question remaining is can the 
associated eigenvalue problem be solved explicitly.

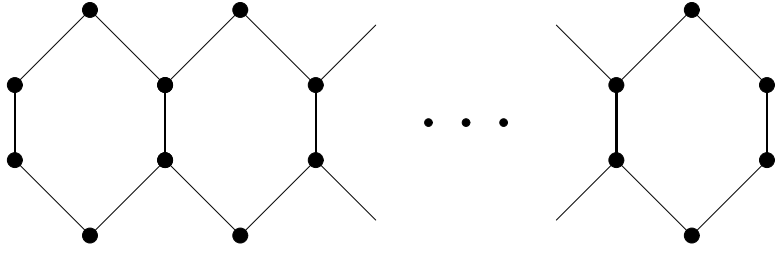
\begin{figure}
\setlength{\unitlength}{1.0mm}
\begin{center}
\begin{picture}(100,30)(0,0)
\multiput(0,0)(20,0){2}{\put(0,10){\circle*{2}}
                        \put(0,20){\circle*{2}}
                        \put(10,0){\circle*{2}}
                        \put(10,30){\circle*{2}}
                        \put(20,10){\circle*{2}}
                        \put(20,20){\circle*{2}}
                        \put(0,10){\line(0,1){10}}
                        \put(0,10){\line(1,-1){10}}
                        \put(0,20){\line(1,1){10}}
                        \put(10,0){\line(1,1){10}}
                        \put(10,30){\line(1,-1){10}}
                        \put(20,10){\line(0,1){10}}}
\multiput(80,0)(20,0){1}{\put(0,10){\circle*{2}}
                        \put(0,20){\circle*{2}}
                        \put(10,0){\circle*{2}}
                        \put(10,30){\circle*{2}}
                        \put(20,10){\circle*{2}}
                        \put(20,20){\circle*{2}}
                        \put(0,10){\line(0,1){10}}
                        \put(0,10){\line(1,-1){10}}
                        \put(0,20){\line(1,1){10}}
                        \put(10,0){\line(1,1){10}}
                        \put(10,30){\line(1,-1){10}}
                        \put(20,10){\line(0,1){10}}}

\put(40,10){\line(1,-1){8}}
\put(40,20){\line(1,1){8}}
\put(80,10){\line(-1,-1){8}}
\put(80,20){\line(-1,1){8}}
\put(55,15){\circle*{1}}
\put(60,15){\circle*{1}}
\put(65,15){\circle*{1}}

\end{picture}
\end{center}
\caption{A bipartite graph $H_{m}$ obtained by gluing $m$ hexagons.}
\label{figure_hexagons}
\end{figure}

Keeping $n$ fixed we know that
${\cal E}(P_{n} \Box P_{m})
\sim
c_{(n)} \cdot (n \cdot m)! \cdot \lambda_{(n)}^{m-1}$
for a constant $c_{(n)}$ and the largest eigenvalue $\lambda_{(n)}$.
Can anything be determined about the sequence
$\lambda_{(n)}$?
What can be said about the asymptotics
of the Euler number
${\cal E}(P_{m} \Box P_{m})$
as $m$ tends to infinity?

A different direction is to study the descent number of directed graphs
(digraphs). For a digraph $G = (V,E)$ on $n$ vertices define its
{\em descent number} to be the number of labelings $\pi$ of the vertices
with $1$ through $n$ such that for each directed edge $u \rightarrow
v$ we have that $\pi(u) < \pi(v)$. If the digraph contains
a directed cycle then the descent number is zero. For an acyclic
digraph (digraphs without directed cycles) the descent number is
strictly positive. 
The classical descent set statistics for permutations
is obtained be looking at orientations of the path.
By gluing directed graphs together, one obtains
classes of graphs whose asymptotics of the descent number
is natural to study via linear operators and their eigenvalues.

The technique of translating
a combinatorial problem into a problem of studying
an operator and its spectrum was also
explored in~\cite{Ehrenborg_Kitaev_Perry},
where consecutive pattern avoiding in permutations were studied.

Finally, we end with a purely enumerative 
question for trees (connected graphs without cycles).
\begin{conjecture}
For a tree $T$ on $n$ vertices the classical Euler number $E_{n}$
is a lower bound for ${\cal E}(T)$, that is,
$$  {\cal E}(T) \geq E_{n} .  $$
Furthermore, equality only holds when the tree $T$
is the path $P_{n}$.
\end{conjecture}

\section*{Acknowledgments}

The first author thanks Bob Strichartz
and the Department of Mathematics at MIT where this paper
was completed.
The first author was partially
supported by
National Security Agency grant H98230-06-1-0072.

\newcommand{\journal}[6]{{\sc #1,} #2, {\it #3} {\bf #4} (#5), #6.}
\newcommand{\book}[4]{{\sc #1,} ``#2,'' #3, #4.}
\newcommand{\bookf}[5]{{\sc #1,} ``#2,'' #3, #4, #5.}
\newcommand{\thesis}[4]{{\sc #1,} ``#2,'' Doctoral dissertation, #3, #4.}
\newcommand{\springer}[4]{{\sc #1,} ``#2,'' Lecture Notes in Math.,
                          Vol.\ #3, Springer-Verlag, Berlin, #4.}
\newcommand{\preprint}[3]{{\sc #1,} #2, preprint #3.}
\newcommand{\preparation}[2]{{\sc #1,} #2, in preparation.}
\newcommand{\appear}[3]{{\sc #1,} #2, to appear in {\it #3}}
\newcommand{\submitted}[4]{{\sc #1,} #2, submitted to {\it #3}, #4.}
\newcommand{\JCTA}{J.\ Combin.\ Theory Ser.\ A}

\newcommand{\communication}[1]{{\sc #1,} personal communication.}

\bigskip
\noindent
{\em R.\ Ehrenborg,
Department of Mathematics,
University of Kentucky,
Lexington, KY~40506-0027}, \\
{\tt jrge@ms.uky.edu} \\

\noindent
{\em Y.\ Farjoun,
Department of Mathematics,
MIT,
Cambridge, MA~02139-4307}, \\
{\tt yfarjoun@math.mit.edu}.


{\small

\begin{thebibliography}{99}

\bibitem{Ehrenborg_Kitaev_Perry}
\preprint{R.\ Ehrenborg, S.\ Kitaev and P.\ Perry}
         {A spectral approach to consecutive pattern avoiding permutations}
         {2007}


\bibitem{Ehrenborg_Levin_Readdy}
\journal{R.\ Ehrenborg, M.\ Levin and M.\ Readdy}
        {A probabilistic approach to the descent statistic}
        {\JCTA}
        {98}{2002}{150--162}


\bibitem{Krein_Rutman}
{\sc M.\ G.\ Kre\u{\i}n and M.\ A.\ Rutman,}
Linear operators leaving invariant a cone in a Banach space.
\emph{Uspehi Matem. Nauk (N.S.)}
\textbf{3} (1948), no.\ 1 (23), 3--95.
English translation in \emph{Amer.\ Math.\ Soc.\ Translation}
\textbf{1950} (1950), no.\ 26, 1--128.

\bibitem{MacMahon}
\book{P.\ A.\ MacMahon}
     {Combinatory Analysis, Vol. I}
     {Chelsea Publishing Company, New York}
     {1960}

\end{thebibliography}

}
\end{document}